\documentclass[10pt,a4paper]{article}
\usepackage{latexsym,amsfonts,amsthm,amsmath,amssymb,cases}
\usepackage{cite}
\usepackage{enumerate}
\usepackage{graphicx}
\usepackage{color}
\usepackage{bm}

\numberwithin{equation}{section}
\newtheorem{thm}{Theorem}[section]

\newtheorem{lemma}{Lemma}[section]

\newtheorem{remark}{Remark}[section]

\newtheorem{defn}{Definition}[section]

\newcommand{\ds}{\displaystyle}
\def\mathrm{\mbox}

\numberwithin{remark}{section}

\begin{document}
%\begin{CJK}{GBK}{kai}
\title{{\Large \bf Blowup time and blowup rate for a pseudo-parabolic equation with singular potential}\thanks{This work was supported by the National Natural Science Foundation of China (12171339).}}
\author{Xiang-kun Shao$^a$, Nan-jing Huang$^a $\footnote{Corresponding author.  E-mail addresses: nanjinghuang@hotmail.com; njhuang@scu.edu.cn} and Donal O'Regan$^b$\\
{\small\it a. Department of Mathematics, Sichuan University, Chengdu, Sichuan 610064, P.R. China}\\
{\small\it b. School of Mathematical and  Statistical Sciences,  University of Galway, Ireland}}
\date{}
\maketitle \vspace*{-11mm}
\begin{center}
\begin{minipage}{6in}
\noindent{\bf Abstract.} This paper provides the upper and lower bounds of blowup time and blowup rate as well as the exponential growth estimate of blowup solutions for a pseudo-parabolic equation with singular potential. %A new upper bound of blowup time is presented and compared to previously proposed bound.
These results complement the ones obtained in the previous literature.
\end{minipage}
\end{center}

\section{Introduction}
\quad\quad In this paper, we focus on the study of the following IBVP of pseudo-parabolic equation with singular potential:
\begin{align}\label{1.1}
 \left\{\begin{array}{ll}
     \ds \frac{u_{t}}{|x|^s}-\Delta u-\Delta u_t=|u|^{p-2}u,\quad &x\in\Omega,\ t>0,\\
     \ds u(x,t)=0,\quad &x\in\partial\Omega,\ t>0,\\
     \ds u(x,0)=u_0(x),\quad &x\in\Omega,
    \end{array}\right.
\end{align}
where $\Omega\subset\mathbb{R}^n(n>2)$ is a bounded domain with smooth boundary $\partial\Omega$, $s\in[0,2]$, $p\in(2,{2n}/(n-2))$ and $u_0\in H_0^1(\Omega)$.

In recent years, the semilinear pseudo-parabolic equation
\begin{equation}\label{1.2}
u_t-\Delta u-k\Delta u_t=f(u),
\end{equation}
being used to describe the deformation of a thermoelastic body, has attracted a lot of attention from scholars (\cite{MR2523294,MR3987484,MR3372307,MR3745341,MR3045640,MR2981259}). For instance, Cao et al. \cite{MR2523294} investigated the existence, uniqueness and large time behavior of solutions for the Cauchy problem of \eqref{1.2} when $k>0$ and $f(u)=u^q$ with $q>0$, and they determined the critical global existence exponent and critical Fujita exponent; Yang et al. \cite{MR2981259} further obtained the second critical exponent for the Cauchy problem of \eqref{1.2} when $k>0$ and $f(u)=u^q$ with $q>1$; Xu and Su \cite{MR3045640} provided the existence of global and blowup solutions for the IBVP of \eqref{1.2} when $k=1$ and $f(u)=u^q$ with $1<q<\infty$ if $n=1,2$ as well as $1<q\leq(n+2)/(n-2)$ if $n>2$; Luo \cite{MR3372307} investigated the finite time blowup and bounds of blowup time for the IBVP of \eqref{1.2} when $k=1$ and $f(u)=u^q$ with $q>1$; Later on, Xu and Zhou \cite{MR3745341} proved a new blowup condition and a new upper bound of blowup time for the IBVP of \eqref{1.2} when $k=1$ and $f(u)=|u|^{q-1}u$ with $q>1$; When $k=1$ and $f$ satisfies some general assumptions, Han \cite{MR3987484} established the existence of blowup solutions of the IBVP of \eqref{1.2} and estimated the upper bound of blowup time.

In fluid mechanics, it is noticed that $|x|^{-s}$ may appear in non-Newton filtration equation describing the motion of a non-Newton fluid in a rigid porous medium under certain assumptions, representing a special medium void (see \cite{MR2036070,MR3162384}). Considering the porous media in the thermoelastic model \eqref{1.2} with $k=1$ and $f(u)=|u|^{p-2}u$, Lian et al. \cite{MR4104462} recently proved the existence and asymptotic behavior of global solutions for \eqref{1.1} at low and critical initial energies. Moreover, they showed the existence of blowup solutions with three initial energy levels and derived the upper bound of blowup time.

It is worth noting that blowup time and blowup rate are indispensable aspects in blowup theory and have always been of interest to researchers (see \cite{MR4728766,MR4705537}). In view of the results obtained in \cite{MR4104462}, we find that the blowup rate, the growth estimate and the lower bound of blowup time of blowup solutions to \eqref{1.1} have not been studied.  Therefore, the purpose of the present paper is to complement the blowup results in \cite{MR4104462}. Firstly, we provide a different method from the concavity method used in \cite{MR4104462} to prove the existence of blowup solutions at low initial energy. Based on it, we determine the upper bound of blowup rate and derive the exponential growth estimate of blowup solutions. Also, we present a new upper bound for blowup time, which is better than the one in \cite{MR4104462} in some specific circumstances. Finally, we estimate the lower bounds of blowup time and blowup rate when a blowup occurred.

The rest of this paper is organized as follows: Section 2 contains some symbols and lemmas. Section 3 contains the statement of main results and their proofs.

\section{Preliminaries}
\quad\quad Throughout this paper, we denote the $L^2(\Omega)$-inner product by $(\cdot,\cdot)$ and the $L^q(\Omega)$-norm ($1\leq q\leq\infty$) by $\|\cdot\|_q$.
For $\varphi\in H_0^1(\Omega)$, we define the functionals
\begin{equation}\label{JI}
J(\varphi):=\frac{1}{2}\|\nabla \varphi\|_2^2-\frac{1}{p}\|\varphi\|_p^p,\quad
I(\varphi):=\|\nabla \varphi\|_2^2-\|\varphi\|_p^p,
\end{equation}
and the mountain pass level
$
d:=\inf_{\varphi\in\mathcal{N}} J(\varphi),
$
where $\mathcal{N}=\{\varphi\in H_0^1(\Omega)\setminus\{0\}\mid I(\varphi)=0\}$ denotes the Nehari manifold. From \cite[Lemma 2.2]{MR4104462}, we have
\begin{equation}\label{djingquezhi}
d=\frac{p-2}{2p}C_*^{-\frac{2p}{p-2}},
\end{equation}
where $C_*>0$ is the best constant satisfying
\begin{equation}\label{C*}
\|\varphi\|_p\leq C_*\|\nabla\varphi\|_2,\quad \forall\varphi\in H_0^1(\Omega).
\end{equation}
By \cite[Remark 1]{MR4104462}, we know that there exists a constant $C_{**}>0$ such that
\begin{equation}\label{C**}
\int_\Omega\frac{\varphi^2}{|x|^s}dx\leq C_{**}\|\nabla\varphi\|_2^2,\quad\forall\varphi\in H_0^1(\Omega).
\end{equation}

The definition of weak solutions is as follows.
\begin{defn}\label{weaksolution}
{\bf (Weak solution)} A function $u\in L^\infty\left(0,T;H_0^1(\Omega)\right)$ with $u_t\in L^2(0,T;H_0^1(\Omega))$ is called a weak solution of \eqref{1.1} over $\Omega\times[0,T)$, if $u(x,0)=u_0(x)\in H_0^1(\Omega)$ and the equality
\begin{equation*}
(|x|^{-s}u_t,\phi)+(\nabla u,\nabla \phi)+(\nabla u_t,\nabla \phi)=(|u|^{p-2}u,\phi)
\end{equation*}
holds for a.e. $t\in(0,T)$ and any $\phi\in H_0^1(\Omega)$.
\end{defn}

Let $u=u(t)$, $t\in[0,T)$ be the weak solution of \eqref{1.1} where $T$ denotes the maximum existence time in this and next sections. Lemma \ref{le1} below is derived directly by equations $(2.2)$, $(4.12)$ and $(4.16)$ of \cite{MR4104462}.

\begin{lemma}\label{le1}
For $t\in[0,T)$, we have that
\begin{equation*}
\begin{split}
&J(u)+\int_0^t\|\nabla u_\tau\|_2^2d\tau+\int_0^t\int_\Omega\frac{u_\tau^2}{|x|^s}dxd\tau=J(u_0),\\
&\frac{1}{2}\frac{d}{dt}\left(\int_\Omega\frac{u^2}{|x|^s}dx+\|\nabla u\|_2^2\right)=-I(u).
\end{split}\end{equation*}
\end{lemma}

The next lemma is crucial in studying the upper bound of blowup rate.

\begin{lemma}\label{le2}
If $I(u_0)<0$ and $0\leq J(u_0)<d$, then there exists a constant $\theta_2>\theta_1:=C_*^{-\frac{2}{p-2}}$ such that
\begin{equation}\label{9}
\|u\|_{p}\geq\theta_2,\quad \|\nabla u\|_2\geq\frac{\theta_2}{C_*},\quad t\in[0,T),
\end{equation}
where $C_*>0$ is defined in \eqref{C*}. Additionally, we have that
\begin{equation}\label{theta0}
\frac{\theta_2}{\theta_1}\geq\theta_0:=\left(\frac{p}{2}-pC_*^{\frac{2p}{p-2}}J(u_0)\right)^{\frac{1}{p-2}}>1.
\end{equation}
\end{lemma}
\begin{proof}
Noting that $I(u_0)<0$ and \eqref{C*}, we obtain $\|u_0\|_p^p>\|\nabla u_0\|_2^2\geq C_*^{-2}\|u_0\|_p^2$. Thus, $\|u_0\|_p>\theta_1:=C_*^{-\frac{2}{p-2}}$.
From \eqref{JI} and \eqref{C*}, one has
\begin{equation}\label{le21}
J(u)\geq \frac{1}{2C_*^2}\|u\|_p^2-\frac{1}{p}\|u\|_p^p.
\end{equation}
Define $h(\theta)=\frac{1}{2C_*^2}\theta^2-\frac{1}{p}\theta^p$ for all $\theta\geq0$. Then $h(\theta)$ is decreasing on $[\theta_1,\infty)$, increasing on $[0,\theta_1]$, and $h(\theta_1)=\frac{p-2}{2p}C_*^{-\frac{2p}{p-2}}=d$.

Owing to $J(u_0)<d$ and \eqref{le21}, there exists a constant $\theta_2>\theta_1$ such that $h(\theta_2)=J(u_0)\geq h(\|u_0\|_p)$, which, together with $\|u_0\|_p>\theta_1$, yields $\|u_0\|_p\geq\theta_2$. Next we prove $\|u\|_p\geq\theta_2$ for $t\in(0,T)$. Indeed, if not, then by $\theta_1<\theta_2$, one can choose a $t_0\in(0,T)$ such that $\theta_1<\|u(t_0)\|_p<\theta_2$. According to \eqref{le21}, we obtain $J(u_0)=h(\theta_2)<h(\|u(t_0)\|_p)\leq J(u(t_0))$, which contradicts that $J(u_0)\geq J(u(t_0))$ (see Lemma \ref{le1}). In addition, we deduce from \eqref{C*} that $C_*\|\nabla u\|_2\geq\|u\|_p\geq\theta_2$. Therefore, \eqref{9} holds.

Next, we show that \eqref{theta0} is true. Let $\theta_3=\theta_2/\theta_1>1$. Then we infer from $h(\theta_2)=J(u_0)$ and $\theta_1=C_*^{-\frac{2}{p-2}}$ that
\begin{equation*}
J(u_0)=h(\theta_1\theta_3)=\theta_1^2\theta_3^2\left(\frac{1}{2C_*^2}-\frac{\theta_1^{p-2}\theta_3^{p-2}}{p}\right)
=\theta_3^2C_*^{-\frac{2p}{p-2}}\left(\frac{1}{2}-\frac{\theta_3^{p-2}}{p}\right).
\end{equation*}
Thanks to $J(u_0)\geq0$ and $\theta_3>1$, we derive
$
\frac{1}{2}-\frac{\theta_3^{p-2}}{p}=\frac{C_*^{\frac{2p}{p-2}}J(u_0)}{\theta_3^2}\leq C_*^\frac{2p}{p-2}J(u_0).
$
Since $J(u_0)<d=\frac{p-2}{2p}C_*^{-\frac{2p}{p-2}}$, one has
\begin{equation*}
\frac{\theta_2}{\theta_1}=\theta_3\geq\left(\frac{p}{2}-pC_*^{\frac{2p}{p-2}}J(u_0)\right)^{\frac{1}{p-2}}
>\left(\frac{p}{2}-\frac{p-2}{2}\right)^{\frac{1}{p-2}}=1,
\end{equation*}
which gives \eqref{theta0}.
\end{proof}

\section{Main results}

\quad\quad In this section, the main results of this paper and their proofs will be stated.
The following theorem is mainly about the upper bounds on blowup rate and time as well as the exponential growth estimate of blowup solutions.

\begin{thm}\label{baopo}
Let $u_0\in H_0^1(\Omega)$. If $J(u_0)<d$ and $I(u_0)<0$, then $u$ blows up in the finite time $T$ in the sense of
\begin{equation}\label{bp}
\lim_{t\rightarrow T^-}\left(\int_\Omega\frac{u^2}{|x|^s}dx+\|\nabla u\|_2^2\right)=\infty.
\end{equation}
Furthermore,
\begin{description}
  \item{$(i)$} the upper bounds of blowup time and rate are given by
  \begin{equation}\label{uptime}
  T\leq\frac{1}{C_1(C_1-2)\mathcal{G}(0)}\left(\int_\Omega\frac{u_0^2}{|x|^s}dx+\|\nabla u_0\|_2^2\right),
  \end{equation}
  and
  \begin{equation}\label{uprate}
  \int_\Omega\frac{u^2}{|x|^s}dx+\|\nabla u\|_2^2\leq
               \frac{1}{2}\left[\frac{C_1(C_1-2)\mathcal{G}(0)}{\left(\int_\Omega\frac{u_0^2}{|x|^s}dx+\|\nabla u_0\|_2^2\right)^{\frac{C_1}{2}}}\right]^{\frac{2}{2-C_1}}(T-t)^{-\frac{2}{C_1-2}},
  \end{equation}
  where
   \begin{equation*}
  \mathcal{G}(0):=\left\{
               \begin{array}{ll}
               \ds  -J(u_0),&\hbox{~if~}J(u_0)<0; \\
               \vspace{-0.05in}\\
               \ds  d-J(u_0),&\hbox{~if~}0\leq J(u_0)<d
               \end{array}
             \right.
  \end{equation*}
  and
  \begin{equation}\label{CC1}
  C_1:=\left\{
               \begin{array}{ll}
               \ds  p,&\hbox{~if~}J(u_0)<0; \\
               \vspace{-0.05in}\\
               \ds  \frac{\left(\theta_0^{p}-1\right)(p-2)}{\theta_0^{p}}+2,&\hbox{~if~}0\leq J(u_0)<d.
               \end{array}
             \right.
  \end{equation}
  Here, $C_*$ and $\theta_0$ are given in \eqref{C*} and \eqref{theta0}, respectively.
  \item{$(ii)$} the exponential growth estimate of blowup solutions is given by
  \begin{equation*}
  \int_\Omega\frac{u^2}{|x|^s}dx+\|\nabla u\|_2^2\geq\left(\int_\Omega\frac{u_0^2}{|x|^s}dx+\|\nabla u_0\|_2^2\right)e^{C_2t},
  \end{equation*}
  where
  \begin{equation}\label{CC2}
  C_2:=\left\{
               \begin{array}{ll}
               \ds  \frac{p-2}{C_{**}+1},&\hbox{~~~if~}J(u_0)<0; \\
               \vspace{-0.05in}\\
               \ds  \frac{(p-2)\left(\theta_0^2-1\right)}{\theta_0^2(C_{**}+1)},&\hbox{~~~if~}0\leq J(u_0)<d.
               \end{array}
             \right.
  \end{equation}
  Here, $C_{**}$ and $\theta_0$ are given in \eqref{C**} and \eqref{theta0}, respectively.
\end{description}
\end{thm}

\begin{proof}
The proof is divided into two parts.

{\bf Part 1: Blowup and the upper bounds of blowup time and rate.}
For $t\in[0,T)$, we set
\begin{equation}\label{H}
\mathcal{H}(t)=\frac{1}{2}\left(\int_\Omega\frac{u^2}{|x|^s}dx+\|\nabla u\|_2^2\right)
\end{equation}
and
\begin{equation}\label{GG}
  \mathcal{G}(t)=\left\{
               \begin{array}{ll}
               \ds  -J(u),&\hbox{~if~}J(u_0)<0; \\
               \vspace{-0.05in}\\
               \ds  d-J(u),&\hbox{~if~}0\leq J(u_0)<d,
               \end{array}
             \right.
  \end{equation}
 If $J(u_0)<0$, then we obtain from Lemma \ref{le1} and \eqref{JI} that
\begin{equation}\label{th11}
0<\mathcal{G}(0)\leq \mathcal{G}(t)=-\frac{1}{2}\|\nabla u\|_2^2+\frac{1}{p}\|u\|_{p}^{p}\leq\frac{1}{p}\|u\|_{p}^{p}.
\end{equation}
In accordance with \eqref{H}, Lemma \ref{le1}, \eqref{JI}, \eqref{GG} and \eqref{th11}, one has
\begin{equation}\label{th12}
\mathcal{H}'(t)=-I(u)=\frac{p-2}{p}\|u\|_{p}^{p}-2J(u)\geq p\mathcal{G}(t).
\end{equation}
If $0\leq J(u_0)<d$, then we deduce from \eqref{GG}, Lemma \ref{le1}, \eqref{JI}, \eqref{djingquezhi}, \eqref{C*} and Lemma \ref{le2} that
\begin{equation}\label{th13}
\begin{split}
0<\mathcal{G}(0)\leq \mathcal{G}(t)&=d-\frac{1}{2}\|\nabla u\|_2^2+\frac{1}{p}\|u\|_{p}^{p}
\leq\frac{p-2}{2p}C_*^{-\frac{2p}{p-2}}-\frac{1}{2C_*^2}\|u\|_{p}^2+\frac{1}{p}\|u\|_{p}^{p}\\
&\leq\frac{p-2}{2p}C_*^{-\frac{2p}{p-2}}-\frac{\theta_1^2}{2C_*^2}+\frac{1}{p}\|u\|_{p}^{p}
=\frac{p-2}{2p}C_*^{-\frac{2p}{p-2}}-\frac{1}{2}C_*^{-\frac{2p}{p-2}}+\frac{1}{p}\|u\|_{p}^{p}
\leq\frac{1}{p}\|u\|_{p}^{p}.
\end{split}\end{equation}
Noting that \eqref{H}, Lemma \ref{le1}, \eqref{JI}, \eqref{GG}, \eqref{djingquezhi}, Lemma \ref{le2} and \eqref{th13}, we obtain
\begin{equation}\label{th14}
\begin{split}
\mathcal{H}'(t)=-I(u)&=\frac{p-2}{p}\|u\|_{p}^{p}-2J(u)=\frac{p-2}{p}\|u\|_{p}^{p}+2\mathcal{G}(t)-2d\\
&=\frac{p-2}{p}\|u\|_{p}^{p}+2\mathcal{G}(t)-\frac{p-2}{p}C_*^{-\frac{2p}{p-2}}
=\frac{p-2}{p}\|u\|_{p}^{p}+2\mathcal{G}(t)-\frac{p-2}{p}\left(\frac{\theta_1}{\theta_2}\right)^{p}\theta_2^{p}\\
&\geq\frac{p-2}{p}\|u\|_{p}^{p}+2\mathcal{G}(t)-\frac{p-2}{p\theta_0^{p}}\|u\|_{p}^{p}\geq \left[\frac{\left(\theta_0^{p}-1\right)(p-2)}{\theta_0^{p}}+2\right]\mathcal{G}(t).
\end{split}\end{equation}

Define $C_1$ as in \eqref{CC1}. From \eqref{th11}, \eqref{th12}, \eqref{th13} and \eqref{th14}, we arrive at
\begin{equation}\label{th15}
\mathcal{H}'(t)\geq C_1\mathcal{G}(t)>0.
\end{equation}
By Lemma \ref{le1}, one has
\begin{equation}\label{th16}
\mathcal{G}'(t)=\int_\Omega\frac{u_t^2}{|x|^s}dx+\|\nabla u_t\|_2^2,
\end{equation}
In accordance with \eqref{th16}, \eqref{H}, Cauchy-Schwarz's inequality and \eqref{th15}, we reach
\begin{equation*}\begin{split}
\mathcal{H}(t)\mathcal{G}'(t)&=\frac{1}{2}\left(\int_\Omega\frac{u^2}{|x|^s}dx+\|\nabla u\|_2^2\right)\left(\int_\Omega\frac{u_t^2}{|x|^s}dx+\|\nabla u_t\|_2^2\right)\\
&\geq\frac{1}{2}\left(\int_\Omega\frac{uu_t}{|x|^s}dx\right)^2+\frac{1}{2}\left(\int_\Omega\nabla u\nabla u_tdx\right)^2+\int_\Omega\frac{uu_t}{|x|^s}dx\int_\Omega\nabla u\nabla u_tdx\\
&=\frac{1}{2}(\mathcal{H}'(t))^2\geq\frac{C_1}{2}\mathcal{H}'(t)\mathcal{G}(t),
\end{split}\end{equation*}
which implies $\frac{\mathcal{G}'(t)}{\mathcal{G}(t)}\geq\frac{C_1\mathcal{H}'(t)}{2\mathcal{H}(t)}$. Integrating from $0$ to $t$ on both sides and making use of \eqref{th15}, we have
\begin{equation}\label{th18}
\frac{\mathcal{H}'(t)}{(\mathcal{H}(t))^{\frac{C_1}{2}}}\geq\frac{C_1\mathcal{G}(0)}{(\mathcal{H}(0))^{\frac{C_1}{2}}}.
\end{equation}
Integrating \eqref{th18} from $0$ to $t$ yields
\begin{equation}\label{th19}
(\mathcal{H}(t))^{\frac{2-C_1}{2}}\leq(\mathcal{H}(0))^{\frac{2-C_1}{2}}-\frac{C_1\mathcal{G}(0)(C_1-2)}{2(\mathcal{H}(0))^{\frac{C_1}{2}}}t.
\end{equation}
Owing to $C_1>2$, \eqref{th19} cannot hold for all $t\geq0$. Thus, $u$ blows up in the finite time $T$, namely,
$\lim\limits_{t\rightarrow T^-}\mathcal{H}(t)=\infty$, where
\begin{equation*}
T\leq\frac{2\mathcal{H}(0)}{C_1\mathcal{G}(0)(C_1-2)},
\end{equation*}
which gives \eqref{uptime}. Moreover, integrating \eqref{th18} from $t$ to $T$ and using $\lim\limits_{t\rightarrow T^-}\mathcal{H}(t)=\infty$, we reach
\begin{equation*}
\mathcal{H}(t)\leq\left[\frac{C_1\mathcal{G}(0)(C_1-2)}{2(\mathcal{H}(0))^{\frac{C_1}{2}}}\right]^{\frac{2}{2-C_1}}(T-t)^{-\frac{2}{C_1-2}},
\end{equation*}
which gives \eqref{uprate}.

{\bf Part 2: Exponential growth estimate.}
If $J(u_0)<0$, then by \eqref{H}, Lemma \ref{le1}, \eqref{JI} and \eqref{C**}, it is clear that
\begin{equation*}
\mathcal{H}'(t)=-I(u)\geq\frac{p-2}{2}\|\nabla u\|_2^2-pJ(u_0)>\frac{p-2}{2}\|\nabla u\|_2^2\geq\frac{p-2}{C_{**}+1}\mathcal{H}(t).
\end{equation*}
If $0\leq J(u_0)<d$, then we obtain from \eqref{H}, Lemma \ref{le1}, \eqref{JI}, \eqref{djingquezhi}, Lemma \ref{le2} and \eqref{C**} that
\begin{equation*}\begin{split}
\mathcal{H}'(t)&=-I(u)\geq\frac{p-2}{2}\|\nabla u\|_2^2-pJ(u_0)>\frac{p-2}{2}\|\nabla u\|_2^2-\frac{p-2}{2}C_*^{-\frac{2p}{p-2}}\\
&=\frac{p-2}{2}\left(\frac{\theta_1^2}{\theta_2^2}\|\nabla u\|_2^2+\frac{\theta_2^2-\theta_1^2}{\theta_2^2}\|\nabla u\|_2^2\right)
-\frac{p-2}{2C_*^2}\theta_1^2\geq\frac{(p-2)\left(\theta_0^2-1\right)}{2\theta_0^2}\|\nabla u\|_2^2\geq\frac{(p-2)\left(\theta_0^2-1\right)}{\theta_0^2(C_{**}+1)}\mathcal{H}(t).
\end{split}\end{equation*}
Then $\mathcal{H}'(t)>C_2\mathcal{H}(t)$, where $C_2$ is defined as in \eqref{CC2}. By \eqref{H}, one has
\begin{equation*}
\int_\Omega\frac{u^2}{|x|^s}dx+\|\nabla u\|_2^2\geq\left(\int_\Omega\frac{u_0^2}{|x|^s}dx+\|\nabla u_0\|_2^2\right)e^{C_2t}.
\end{equation*}
\end{proof}

\begin{remark}
\begin{description}
  \item{$(i)$} In the case of $J(u_0)<0$, the upper bound of blowup time in our theorem is equal to the one in \cite[Theorem 4.11]{MR4104462}.
  \item{$(ii)$} In the case of $0<J(u_0)<d$, we conclude that the upper bound of blowup time in our theorem is less than that in \cite[Theorem 4.13]{MR4104462} if $\left[\frac{(1-\varepsilon)p+2\varepsilon}{2}\right]^{-\frac{p}{p-2}}<\frac{p-1}{p-2}-\left[\frac{1}{(p-2)^2}+\frac{p}{4(p-1)}\right]^{\frac{1}{2}}$, where $\varepsilon\in\left[\frac{J(u_0)}{d},1\right)$.
\end{description}
\end{remark}

The next theorem is about the lower bounds of blowup time and rate.

\begin{thm}\label{xiajie}
Let $u_0\in H_0^1(\Omega)$. If the weak solution $u$ blows up in the finite time $T$ in the sense of \eqref{bp}, then
\begin{equation*}
T\geq\frac{1}{C_*^{p}(p-2)}\left(\int_\Omega\frac{u_0^2}{|x|^s}dx+\|\nabla u_0\|_2^2\right)^{\frac{2-p}{2}},
\end{equation*}
and
\begin{equation}\label{lorate}
\int_\Omega\frac{u^2}{|x|^s}dx+\|\nabla u\|_2^2\geq\frac{1}{C_*^{\frac{2p}{p-2}}(p-2)^{\frac{2}{p-2}}}(T-t)^{-\frac{2}{p-2}},
\end{equation}
where $C_*>0$ is given in \eqref{C*}.
\end{thm}

\begin{proof}
Let $\mathcal{H}$ be defined as in \eqref{H}. From the proof of Theorem \ref{baopo}, we have $\lim\limits_{t\rightarrow T^-}\mathcal{H}(t)=\infty$. According to Lemma \ref{le1}, \eqref{JI} and \eqref{C*}, we discover
\begin{equation*}
\mathcal{H}'(t)=-\|\nabla u\|_2^2+\|u\|_{p}^{p}\leq C_*^{p}\|\nabla u\|_2^{p}\leq 2^{\frac{p}{2}}C_*^{p}(\mathcal{H}(t))^{\frac{p}{2}},
\end{equation*}
which gives
\begin{equation}\label{th21}
\frac{\mathcal{H}'(t)}{(\mathcal{H}(t))^{\frac{p}{2}}}\leq2^{\frac{p}{2}}C_*^{p}.
\end{equation}
Integrating it from $0$ to $t$ gives
\begin{equation*}
\frac{2}{2-p}(\mathcal{H}(t))^{\frac{2-p}{2}}\leq\frac{2}{2-p}(\mathcal{H}(0))^{\frac{2-p}{2}}+2^{\frac{p}{2}}C_*^{p}t.
\end{equation*}
Letting $t\rightarrow T^-$ and using $\lim\limits_{t\rightarrow T^-}\mathcal{H}(t)=\infty$ yield the lower bound of blowup time
\begin{equation*}
T\geq\frac{2^{\frac{2-p}{2}}}{C_*^{p}(p-2)}(\mathcal{H}(0))^{\frac{2-p}{2}}=\frac{1}{C_*^{p}(p-2)}\left(\int_\Omega\frac{u_0^2}{|x|^s}dx+\|\nabla u_0\|_2^2\right)^{\frac{2-p}{2}}.
\end{equation*}
Integrating \eqref{th21} from $t$ to $T$ and taking $\lim\limits_{t\rightarrow T^-}\mathcal{H}(t)=\infty$ result in the lower bound of blowup rate
\begin{equation*}
\mathcal{H}(t)\geq\frac{1}{2C_*^{\frac{2p}{p-2}}(p-2)^{\frac{2}{p-2}}}(T-t)^{-\frac{2}{p-2}},
\end{equation*}
which means \eqref{lorate}.
\end{proof}

%\bibliographystyle{plain}
%\bibliography{a}

%\end{CJK}

\end{document}